\documentclass[12pt]{amsart}
\usepackage{amsfonts}
\usepackage{amssymb}
\usepackage{amsmath}
\usepackage{amsthm}

\newtheorem{theorem}{Theorem}
\newtheorem{corollary}{Corollary}
\newtheorem{lemma}{Lemma}

\theoremstyle{definition}
\newtheorem{definition}{Definition}

\begin{document}

\title[]{Function Model of the Teichm\"uller space of a closed hyperbolic Riemann Surface}

\author{Yunping Jiang}


\address{Department of Mathematics\\
Queens College of the City University of New York\\
Flushing, NY 11367-1597\\
and\\
Department of Mathematics\\
Graduate School of the City University of New York\\
365 Fifth Avenue, New York, NY 10016}
\email[]{yunping.jiang@qc.cuny.edu}




\begin{abstract}
We introduce a function model for the Teichm\"uller space of a
closed hyperbolic Riemann surface. Then we introduce a new metric
by using the maximum norm on the function space on the
Teichm\"uller space. We prove that the identity map from
the Teichm\"uller space equipped with the usual Teichm\"uller metric
to the Teichm\"uller space equipped with this new metric is uniformly continuous.
Furthermore, we also prove that the inverse of the identity, that is, the identity map from
the Teichm\"uller space equipped with this new metric to the Teichm\"uller space
equipped with the usual Teichm\"uller metric, is continuous.
Therefore, the topology induced by the new metric
is just the same as the topology induced by the usual Teichm\"uller metric
on the Teichm\"uller space. We give a remark about the pressure metric and the Weil-Petersson metric.
\end{abstract}

\maketitle

\section{\bf Introduction}
A closed Riemann surface is a compact connected complex
one-dimensional surface. We only consider an oriented surface. A
topological characterization of a closed Riemann surface is its
genus $g$. Riemann observed that all genus $g=0$ closed Riemann
surfaces are conformally equivalent to the standard Riemann sphere
${\mathbb P}^{1}$. However, this is not in general true for closed
Riemann surfaces of positive genus. Suppose $R$ is a closed
Riemann surface of genus $g\geq 2$. Then it is hyperbolic and
conformally equivalent to the open unit disk modulo a Fuchsian
group. A marked Riemann surface by $R$ is a pair $(X, h)$ where
$h: R\to X$ is an orientation-preserving homeomorphism. In the
space of all marked Riemann surfaces $(X,h)$ by $R$, one can
introduces a conformal equivalence relation. This space modulo
this equivalence relation is called the Techm\"uller space $T(R)$.
In other words, $T(R)$ is the quotient space of all complex
structures on $R$ by those orientation-preserving diffeomorphisms
which are isotopic to the identity. We know that
\begin{itemize}
\item[a)] $T(R)$ is homeomorphic to ${\mathbb R}^{6g-6}$,
\item[b)] $T(R)$ admits a complex manifold structure of $3g-3$,
\item[c)] $T(R)$ can be embedded into ${\mathbb C}^{3g-3}$ as a
contractible set, and
\item[d)] $T(R)$ is a pseudoconvex domain.
\end{itemize}

The Teichm\"uller space $T(R)$ is an important subject in the
modern mathematics and physics. It is a cover of the moduli space
which is the space of all complex structures on $R$ modulo the
action of orientation-preserving diffeomorphisms. The moduli space
has the same dimension $6g-6$. The moduli space was first
considered by Riemann and plays an important role in the modern
string theory.

To understand the Teichm\"uller space $T(R)$, several models have
been introduced. For examples, we have Bers' embedding and
Thurston's embedding. Furthermore, several metrics have been
introduced on the Teichm\"uller space $T(R)$. The first metric
$d_{T}(\cdot, \cdot)$ is introduced by Teichm\"uller. There are
other metrics, for examples, the Kobayashi metric and the
Weil-Petersson metric. Royden proved that the Teichm\"uller metric
is equal to the Kobayashi metric in this case and Gardiner even
generalized this result for any Riemann surfaces of infinite
analytic type (refer to~\cite{GardinerJiangWang} for a proof and
furthermore references).

In this paper, we introduce a new model of the Teichm\"uller space
from the dynamical system point of views. This new model is a
space of functions defined on a Cantor set $\Sigma^{*}_{A}$. The
graphs of these functions are in the infinite-dimenisonal cube
$\prod_{0}^{\infty} (0,1)$ of ${\mathbb R}^{\infty}$. Therefore,
we have the maximum norm on the function space. This maximum norm
introduces a maximum metric $d_{max} (\cdot, \cdot)$ on the
Teichm\"uller space $T(R)$. We prove that the identity map from
the Teichm\"uller space equipped with the usual Teichm\"uller metric
to the Teichm\"uller space equipped with this new metric is uniformly continuous (see
Theorem~\ref{finermetric}). Furthermore, we also prove that the identity map from
the Teichm\"uller space equipped with this new metric to the Teichm\"uller space
equipped with the Teichm\"uller metric is continuous
(see Theorem~\ref{semifinermetric}). Therefore, the topology induced by the new metric
is just the same as the topology induced by the usual Teichm\"uller metric on the Teichm\"uller space.

The paper is organized as follows. In \S2, we define and review
the Teichm\"uller space $T(R)$ of a closed hyperbolic Riemann
surface $R$ and mention a theorem due to Earle and McMullen which
we will used in this paper (see Theorem~\ref{EarleMcMullen}). In
\S3, we use the Nielsen development for Fuchsian groups to
construct expanding transitive Markov maps for any marked Riemann
surfaces by a standard closed hyperbolic Riemann surface. We use
Bowen's paper~\cite{Bowen} and Bowen and Series'
paper~\cite{BowenSeries} as two references. In \S4, we define the
symbolic space $\Sigma_{A}$ and the dual symbolic space
$\Sigma_{A}^{*}$ for all marked Riemann surfaces by a fixed
standard closed hyperbolic Riemann surface. The symbolic space
$\Sigma_{A}$ is treated as the topological model of all such
marked Riemann surfaces. We define geometric models on the dual
symbolic space $\Sigma_{A}^{*}$ for all marked Riemann surfaces by
a fixed standard closed hyperbolic Riemann surface in \S4.2. The
geometric models are Lipschitz continuous functions defined on
$\Sigma_{A}^{*}$. Their graphs are contained in the
infinite-dimensional cube $\prod_{0}^{\infty}(0,1)$. To prove
these functions are geometric models, we mention Tukia's theorem
which is a stronger version than Mostow's rigidity theorem in
$2$-dimensional case. For the sake of completeness of the paper,
we give a proof of Tukia's theorem from the dynamical system point
of views. We call each geometric model a scaling function. We use
${\mathcal F}$ to denote the space of all scaling functions. In
\S5, we prove that there is a one-to-one and onto maps between the
Teichm\"uller space and the function space ${\mathcal F}$. In \S6,
we discuss Bers' embedding and the complex manifold structure on
${\mathcal F}$. Using the maximum norm on the function space
${\mathcal F}$, we define the maximum metric
$d_{max}(\cdot,\cdot)$ on the Teichm\"uller space in \S7.
We prove, in the same section, Theorem~\ref{finermetric}, that is,
the identity map from
the Teichm\"uller space equipped with the usual Teichm\"uller metric
to the Teichm\"uller space equipped with this new metric is uniformly continuous.
Furthermore, we prove, in the same section,
Theorem~\ref{semifinermetric}, that is, that the identity map from
the Teichm\"uller space equipped with this new metric to the Teichm\"uller space
equipped with the Teichm\"uller metric is continuous.
Finally, in \S8, we give a remark to compare our function model and McMullen's thermodynamical embedding in his recent paper~\cite{McMullen}.
Furthermore, by following McMullen's calculation of the Weil-Petersson metric on the tangent space
of the Teichm\"uller space by the pressure metric, we show that the pressure metric of the tangent vector
to any smooth curve at a point in our function model is a constant times the Weil-Petersson metric.

\vspace{20pt} \noindent {\bf Acknowledgement:}
This work is partially done when I visited the Institut des Hautes \'Etudes
Scientifiques in Bures-sur-Yvette, France and when I visited the
Academy of Mathematics and System Science and the Morningside
Center of Mathematics at the Chinese Academy of Sciences in
Beijing, China. I would like to thank these institutions for their
hospitality. This work is partially supported by NSF grants,
PSC-CUNY awards, and the Bai Ren Ji Hua of Chinese Academy of
Sciences. I would like to thank Professors Fred Gardiner and Bill Harvey
for many interesting conversations. I would like to thank Professor Curt McuMullen for his
useful comments and suggestions for the first version of this paper
and for pointing to me his recent paper~\cite{McMullen} which leads to \S8.

\section{\bf Teichm\"uller Space of a closed hyperbolic Riemann surface}

We first discuss the Teichm\"uller space of a closed hyperbolic
Riemann surface $R$. Suppose ${\mathbb D}$ is the unit disk and
suppose $S^{1}=\partial {\mathbb D}$ is the unit circle. Let $R=
{\mathbb D}/\Gamma$ be a closed hyperbolic Riemann surface,
presented as the quotient of the unit disk by a Fuchsian group
whose limit set is the whole $S^{1}$. Any quasiconformal map $h:
R\to R$ can be lift to a quasiconformal map $H: {\mathbb D}\to
{\mathbb D}$. The map $H$ can be extended to a homeomorphism of
$\overline{\mathbb D}={\mathbb D}\cup S^{1}$. The key in defining
the Teichm\"uller space is to know which quasiconformal maps $h :
R\to R$ are to be considered trivial. The following theorem gives
an answer to this question.

\vspace*{10pt}
\begin{theorem}[Earle-McMullen~\cite{EarleMcMullen}]~\label{EarleMcMullen}
Suppose $h : R\to R$ is a quasiconformal map. The following are
equivalent.
\begin{itemize}
\item[1)] There is a lift of $h$ to a map $H: {\mathbb D}\to {\mathbb D}$ that extends to the identity
on $S^{1}$.
\item[2)] The map $h$ is homotopic to the identity rel ideal boundary (in this case the ideal boundary is empty).
\item[3)] The map $h$ is isotopic to the identity rel ideal boundary, through uniformly quasiconformal maps.
\end{itemize}
\end{theorem}

A marked Riemann surface by $R$ is a pair $(X, h_{X})$ where
$h_{X} : R\to X$ is an orientation preserving quasiconformal
homeomorphism. Two marked Riemann surfaces $(X, h_{X})$ and $(Y,
h_{Y})$ are equivalent if there is a conformal isomorphism
$\alpha: X\to Y$ such that
$$
f = h_{Y}^{-1}\circ \alpha\circ h_{X} : R\to R
$$
is isotopic to the identity. The Teichm\"uller space $T(R)$ of $R$
is the space of equivalence classes $[(X,h_{X})]$ of all marked
Riemann surfaces $(X, h_{X})$ by $R$, that is,
$$
T(R) =\{ [(X,h_{X})]\}.
$$

\section{\bf Nielsen development, Markov partition, and expanding transitive Markov map}
Suppose $X$ is a closed hyperbolic Riemann surface. Let $g\geq 2$
be the genus of $X$. Since its universal cover is the unit disk
${\mathbb D}$, so through the universal cover, we can write
$X={\mathbb D}/\Gamma_{X}$ where $\Gamma_{X}$ is a Fuchsian group
whose limit set is the whole $S^{1}$.

\begin{definition}
A piecewise smooth map $f: S^{1}\to S^{1}$ is called Markov for
$\Gamma_{X}$ if we can cut $S^{1}$ into finitely many intervals
$I_{1}$, $\cdots$, $I_{k}$ such that
\begin{itemize}
\item[i)] $S^{1} =\cup_{i=1}^{k} I_{i}$,
\item[ii)] $I_{i}$ and $I_{j}$ have disjoint interiors for any
$1\leq i\not=j\leq k$,
\item[iii)] $f|I_{i} =\gamma_{i}|I_{i}$ for some $\gamma_{i}\in
\Gamma_{X}$, and
\item[iv)] $f(I_{j})$ is the union of some intervals of $I_{i}$'s
for each $1\leq j\leq k$.
\end{itemize}
\end{definition}

Let $W=\cup_{i=1}^{n}\partial I_{i}$. The iv) is equivalent to the
statement that
$$
f(W)\subset W.
$$
For any $1\leq i, j\leq k$, we write $i\to j$ if $f(I_{i})\supset
I_{j}$.

\begin{definition} A Markov map $f: S^{1}\to S^{1}$ for the
surface group $\Gamma_{X}$ is called transitive if for any $1\leq
i,j\leq k$, there are $1\leq i_{0}, i_{1}, \cdots, i_{n}\leq k$
such that
$$
i=i_{0}\to i_{1}\to \cdots \to i_{n} =j.
$$
This is equivalent to say that $f^{n}(I_{i})\supset I_{j}$
\end{definition}

\begin{definition} A Markov map $f$ for $\Gamma_{X}$ is called
expanding if there are two constants $C>0$ and $\lambda
>1$ such that
$$
|(f^{n})'(x)|\geq C\lambda^{n}
$$
for $x\in I_{i}$ and $n\geq 0$.
\end{definition}

Suppose $X_{0}={\mathbb D}/\Gamma_{0}$ is the closed Riemann
surface of genus $g$ such that the fundamental domain $D$ of
$X_{0}$ in ${\mathbb D}$ for $\Gamma_{0}=\Gamma_{X_{0}}$ is a
regular $4g$-sided non-Euclidean polygon. We call $X_{0}$ the
standard Riemann surface of genus $g$.

Each angle of $D$ is $\frac{\pi}{2g}$. Each vertex of $D$ belongs
to $4g$ distinct translations $\gamma (D)$ for $\gamma\in
\Gamma_{0}$. All $\gamma (D)$ for $\gamma\in \Gamma_{0}$ form a
net ${\mathcal R}$ in ${\mathbb D}$. The net ${\mathcal R}$ has
the following property: the entire non-Euclidean geodesic passing
through any edge in the net ${\mathcal R}$ is contained in the net
${\mathcal R}$. Let $V_{0}$ be the set of vertices of $D$. Let $V$
be the set of vertices in the net ${\mathcal R}$ which are
adjacent in ${\mathcal R}$ to $V_{0}$ but not in $V_{0}$. Consider
all polygons $\tilde{D}$ adjacent to $D$. Then $V$ are all
vertices of $\tilde{D}$ minus $V_{0}$. For each vertex $p$ of
${\mathcal R}$, there are $2g$ non-Euclidean geodesic passing
through it. These $2g$ non-Euclidean geodesics have $4g$ endpoints
at infinity. Let $W_{p}$ be the set of these $4g$ points at
infinity. Define
$$
W =\cup_{p\in V} W_{p}.
$$
Then $W_{q}\subset W$ for any vertex $q$ of $D$.

The $4g$ sides of $D$ give a set of generators for $\Gamma_{0}$ as
follows. Divide the sides of $D$ into $g$ groups of $4$
consecutive sides; label the $j^{th}$ group $a_{j}, b_{j},
a_{j}^{-1}, b_{j}^{-1}$. Call $a_{j}$ and $a_{j}^{-1}$ and $b_{j}$
and $b_{j}^{-1}$ corresponding sides. For each side $s$ of $D$
there is an element $\gamma_{s}\in \Gamma_{0}$ such that
$$
\gamma_{s} (s) = D \cap \phi_{s} (D)= \hbox{side corresponding to
$s$}.
$$
The set $\{ \gamma_{s}\}$ generates $\Gamma_{0}$. The
non-Euclidean geodesic passing $s$ cuts $S^{1}$ into two
intervals. Let $J_{s}$ be the smaller one. Then we have that
$$
\gamma_{s} (J_{s}\cap W) \subset W.
$$

Let $v\in V_{0}$. Then we have two sides $s$ and $s'$ belonging to
$v$. Let $\beta$ and $\beta'$ are two non-Euclidean geodesic
passing trough $s$ and $s'$, respectively. Consider the interval
$J_{s}\cap J_{s'}$. Let $p\in \beta$, $p'\in \beta'$ be the
vertices of the net ${\mathcal R}$ adjacent to $v$ in the net. Let
$\gamma(D)$ be the translation of $D$ having $v$, $p$, and $p'$ as
its vertices for some $\gamma\in \Gamma_{0}$. Let $q$, $q'$ are
vertices of $\gamma(D)$ such that $q,p,v,p'q'$ are consecutive
vertices of $\gamma(D)$. The non-Euclidean geodesics $\delta$,
$\delta'$ passing $p,q$ and $p',q'$, respectively, do not
intersect and have points $w(v)$ and $w'(v)$ at infinity in the
interior of $J_{s}\cap J_{s'}$. Let $J(v) =[w(v), w'(v)]$ in the
interior of $J_{s}\cap J_{s'}$. Note that $w(v), w'(v)\in W$ and
$J(v)$ does not contain any other points from $W$.

The set $W$ cuts $S^{1}$ into finitely many intervals $I_{1,0}$,
$\cdots$, $I_{k,0}$. Define the map $f_{0}: S^{1}\to S^{1}$ as
$$
f_{0}|I_{j,0} =\gamma_{s}|I_{j,0}, \quad J_{s}\supset I_{j,0}.
$$
Then $f_{0}$ is a piecewise M\"obius transformations and is Markov
since $f_{0}(W)\subset W$.

Since some $I_{j,0}$'s belong to more than one $J_{s}$, there are
a number of ways to define $f_{0}$. This flexibility allows us to
eventually get an expanding Markov map. Given a vertex $v$ of $D$.
Let $s$ be a side of $D$ having $v$ as their common vertex. Let
$v'$ be the other vertex of $s$. We assume that from $v'$ to $v$
are clockwise. Suppose $\tilde{J}_{s}$ is the maximal interval
where
$$
f_{0}|\tilde{J}_{s}=\gamma_{s}|\tilde{J}_{s}.
$$
We require that
\begin{itemize}
\item[a)] $\hbox{int} J_{s} \supset \tilde{J}_{s} \supset
J_{s}\setminus (J(v)\cup J(v'))$, and
\item[b)] $J(v)\subset \tilde{J}_{s}$ and $\tilde{J}_{s}$
disjoint with the interior $\hbox{int} J(v')$.
\end{itemize}
In other words, suppose $J_{s}$ is cut by $W$ into intervals
clockwise $I_{1,0}$, $I_{2,0}=J(v')$, $I_{3,0}$, $\cdots$,
$I_{l_{s}-1,0}=J(v)$, and $I_{l_{s},0}$. So we define
$\tilde{J}_{s}=\cup_{j=3}^{l_{s}-1} I_{j,0}$.

For each non-Eucildean geodesic $\beta$ passing $s$, it is on the
isometric circle $C_{s}$ of $\gamma_{s}$, i.e.,
$|\gamma_{s}'(x)|=1$ for $x\in \beta$. Inside this circle,
$|\gamma_{s}'(x)|>1$ and outside this circle,
$|\gamma_{s}'(x)|<1$. Since $\tilde{J}_{s}$ is inside this circle,
so $|f_{0}'(x)| >1$ for $x\in \tilde{J}_{s}$. Since there are
finitely many $\tilde{J}_{s}$ and each one is a compact interval,
so there is a constant $\lambda_{0}>1$, such that
$$
|f_{0}'(x)| \geq \lambda_{0},\quad \forall x\in I_{j,0}, \quad
1\leq j\leq k.
$$
This implies that $f_{0}$ is an expanding Markov map for
$\Gamma_{0}$ with the Markov partition
$$
\eta_{0,0} =\{I_{1,0},\cdots, I_{k,0}\}.
$$

Now we prove that the Markov map $f_{0}$ is transitive. First, for
each $1\leq i\leq k$, there is some iterate $f^{n}(I_{i})$
contains $J(v)$ for some vertex $v$ of $D$. This is because that
otherwise, $f$ is continuous on $f^{n}(I_{i})$ and
$f^{n+1}(I_{i})$ is an interval longer than $f^{n}(I_{i})$ because
of the expanding condition. But this can not continuous
indefinitely.

From $\eta_{0,0}$, we can generate a sequence of Markov partitions
$$
\eta_{n,0}=f^{-n}_{0}\eta_{0,0}
$$
for $n=0, 1, \cdots$. (See \S4.1 for more detailed description
about intervals in $\eta_{n}$.) Let
$$
\nu_{n,0} = \max_{I\in \eta_{n,0}} |I|.
$$
Since $|f_{0}'(x)|\geq \lambda_{0}$ for $x\in I_{j,0}$, $1\leq
j\leq k$, we have that
$$
\nu_{n,0} \leq \lambda_{0}^{-n}
$$
for $n=0, 1, \cdots$.

Now we construct a transitive expanding Markov map for any closed
Riemann surface $X={\mathbb D}/\Gamma_{X}$ of genus $g$ associated
to an isomorphism $\phi$. Suppose $\phi: \Gamma_{0}\to \Gamma_{X}$
is an isomorphism. Then there is a unique homeomorphism $H:
S^{1}\to S^{1}$ such that
$$
H(\gamma (x)) =\phi (\gamma) (H(x))
$$
for any $x\in S^{1}$ and any $\gamma\in \Gamma_{0}$. Here $H$ is
called the boundary correspondence. Moreover, $H$ is
quasisymmetric.

Let $I_{j}= I_{j,X}=H(I_{j,0})$. Then $S^1=\cup_{j=1}^{k} I_{j}$.
Define
$$
f_{X} =f_{X,\phi}: S^{1}\to S^{1}
$$
as
$$
f_{X}|I_{j} = \phi (\gamma_{s})|I_{j}, \quad  J_{s} \supset
I_{j,0}.
$$
In other words,
$$
f_{X}= H\circ f_{0}\circ H^{-1}.
$$
Then $f_{X}$ is a transitive Markov map for $\Gamma_{X}$ with the
Markov partition
$$
\eta_{0}=\eta_{0,X, \phi}=\{ I_{1}, \cdots, I_{k}\}.
$$
Furthermore, we can generate a sequence of Markov partitions,
$$
\eta_{n}=\eta_{n,X,\phi}=f_{X}^{-n}\eta_{0}
$$
for $n=0, 1, \cdots$. Let
$$
\nu_{n}=\nu_{n,X,\phi}= \max_{I\in \eta_{n}} |I|
$$
for $n\geq 0$. Since a quasisymmetric homeomorphism is H\"older
continuous, so there are constants $A>0$ and $0<\mu<1$ such that
$$
\nu_{n}\leq A \mu^{n}, \quad \forall\; n\geq 0.
$$
Furthermore, we have that

\begin{lemma}~\label{boundedgeometry} The Markov map $f_{X}$ is expanding and the sequence
of Markov partitions has bounded geometry, that is, there is a
constant $C>0$ such that
$$
\frac{|J|}{|I|} \geq C
$$
for any $J\subset I$ with $J\in\eta_{n+1}$ and $I\in \eta_{n}$.
\end{lemma}

\begin{proof}
Since $f_{X}$ is piecewise M\"obius transformations, so it is a
piecewise $C^{2}$ Markov map. That is
$$
\Big| \log |f_{X}'(x)| -\log |f_{X}'(y)|\Big| \leq \frac{1}{m}
|f_{X}'(x) -f_{X}'(y)| \leq \frac{M}{m} |x-y|
$$
for any $x,y\in I\in\eta_{0}$, where
$$
0< m= \min_{x\in \in I\in\eta_{0}} |f_{X}'(x)|,\quad M=\max_{x\in
I\in \eta_{0}} |f_{X}''(x)|<\infty.
$$

Consider any $x,y\in I\in \eta_{n}$. Then $f_{X}^{i}(x),
f_{X}^{i}(y) \in \eta_{n-i}$ and
$$
\Big| \log \frac{|(f_{X}^{n})'(x)|}{|(f_{X}^{n})'(y)|} \Big| \leq
\sum_{i=0}^{n-1} | \log f_{X}'(f_{X}^{i}(x)) -\log
f_{X}'(f_{X}^{i}(y))|
$$
$$
\leq \frac{M}{m} \sum_{i=0}^{n-1} |f_{X}^{i}(x)-f_{X}^{i}(y)| \leq
\frac{AM}{m} \sum_{i=0}^{n-1} \mu^{n-i} \leq \frac{AM}{m (1-\mu)}.
$$
Thus we get a distortion result,
\begin{equation}~\label{distortion1}
B^{-1}\leq \frac{|(f_{X}^{n})'(x)|}{|(f_{X}^{n})'(y)|} \leq B
\end{equation}
for any $x,y\in I\in \eta_{n}$ and any $n>0$, where
$$
B= \exp \Big( \frac{AM}{m (1-\mu)}\Big).
$$

Now for any $n>0$ and any $J\subset I$ with $J\in\eta_{n+1}$ and
$I\in \eta_{n}$,
$$
\frac{|f_{X}^{n}(J)|}{|f_{X}^{n}(I)|} =
\frac{|(f_{X}^{n})'(x)|}{|(f_{X}^{n})'(y)|} \frac{|J|}{|I|}
$$
for some $x,y\in I$. Let
$$
E= \min_{J\subset I, J\in \eta_{1}, I\in \eta_{0}}
\frac{|J|}{|I|}.
$$
Then
$$
\frac{|J|}{|I|} \geq C=\frac{E}{B}.
$$
This says that the sequence $\{\eta_{n}\}_{n=0}^{\infty}$ of
Markov partitions has bounded geometry.

Now for any $n>0$ and $x\in I\in \eta_{n}$, let $y\in I$ such that
$$
|f_{X}^{n} (I)|= |(f_{X}^{n})'(y)| |I|.
$$
then we have that
$$
|(f_{X}^{n})'(x)| = \frac{|(f_{X}^{n})'(x)|}{|(f_{X}^{n})'(y)|}
\frac{|f_{X}^{n} (I)|}{|I|} \geq \frac{K}{BA} \mu^{-n} =C_{0}
\lambda^{n}
$$
where $K=\min_{I\in \eta_{0}}|I|$ and $C_{0} =K/(BA)$ and $\lambda
=\mu^{-1}$. So $f_{X}$ is expanding.
\end{proof}

\section{\bf Symbolic representation and dual symbolic representation}

\subsection{Topological model.} Given any marked Riemann surface $(X,h_{X})$ by
$X_{0}$. Since $h_{X}: X_{0}\to X$ is an orientation preserving
quasiconformal homeomorphism, it induces an isomorphism $\phi:
\Gamma_{0}\to \Gamma_{X}$. Let $f_{X}: S^{1}\to S^{1}$ be the
transitive expanding Markov map constructed in the previous
section with the initial Markov partition $\eta_{0} =\{
I_{1},\cdots, I_{k}\}$. Associating to $\eta_{0}$, we have a
$k\times k$ $0$-$1$ matrix $A=(a_{ij})_{k\times k}$, where
$a_{ij}=1$ if $f(I_{i})\supset I_{j}$ and $a_{ij}=0$ otherwise.
Let
$$
\Sigma_{A} = \{ w=i_{0}i_{1}\cdots i_{n}i_{n+1}\cdots  \;|\;
i_{n}\in \{1,\cdots, k\}, a_{i_{n}}a_{i_{n+1}} =1, n=0,1,
\cdots\}.
$$
The topology of $\Sigma_{A}$ is given as follows. For any $n\geq
0$, let
$$
w_{n}=i_{0}i_{1}\cdots i_{n}, \quad a_{i_{l}i_{l+1}}=1,\; 0\leq
l\leq n-1.
$$
Define the left cylinder
$$
[w_{n}]=\{ w' =w_{n} i_{n+1}'\cdots i_{n+m}'i_{n+m+1}'\cdots\}
$$
where $i_{n+m}'\in \{1,\cdots, k\}$, $a_{i_{n}i_{n+1}'}=1$ and
$a_{i_{n+m}'i_{n+m+1}'}=1$, $m=0, 1, \cdots$. Then all these left
cylinders form a topological basis. The space $\Sigma_{A}$ with
this topological basis is called the symbolic space for
$(X,h_{X})$.

Let $\sigma_{A}$ be the shift defined as
$$
\sigma_{A}: w = i_{0}i_{1}\cdots i_{n}i_{n+1}\cdots \to
\sigma_{A}(w) =i_{1}\cdots i_{n}i_{n+1}\cdots.
$$
Then $(\Sigma_{A}, \sigma_{A})$ is a sub-shift of finite type.

For any pair $i$ and $j$ such that $a_{ij}=1$, there is an
interval $I_{ij}\in \eta_{1}$ such that $f: I_{ij}\to I_{j}$ is a
homeomorphism. Let $g_{ij}: I_{j}\to I_{ij}$ be its inverse. For
each $w=i_{0}i_{1}\cdots i_{n}i_{n+1}\cdots \in \Sigma_{A}$, let
$w_{n}=i_{0}i_{1}\cdots i_{n}$. Define
$$
g_{w_{n}} = g_{i_{0}i_{1}} \cdots g_{i_{n-1}i_{n}}
$$
and
$$
I_{w_{n}} =g_{w_{n}} (I_{i_{n}}).
$$
Then $I_{w_{n}}\in \eta_{n+1}$ and $g_{w_{n}}: I_{i_{n}}\to
I_{w_{n}}$ is a homeomorphism. One can check that
$$
\cdots\subset I_{w_{n}}\subset I_{w_{n-1}} \subset
I_{w_{1}}\subset I_{i_{0}}
$$
Since the length of $I_{w_{n}}$ tends to zero exponentially as $n$
goes to infinity, the set $\cap_{n=0}^{\infty} I_{w_{n}}$ contains
one point $x_{w}$. Define
$$
\pi (w)=x_{w}.
$$
Then we have that
$$
\pi (\sigma_{A} (w)) = f_{X} (\pi (w)), \quad \forall \; w\in
\Sigma_{A}.
$$
Note that $\pi$ is 1-1 except for countably many points which are
endpoints of $I_{w_{n}}$ for all $w_{n}$, $n\geq 0$.

{\em Thus from the dynamical system point of views, the symbolic
dynamical system $(\Sigma_{A}, \sigma_{A})$ is the topological
model for all marked Riemann surfaces $(X, h_{X})$ by $X_{0}$.}

\subsection{\bf Geometric models.}
Now we are going to define the dual symbolic space
$\Sigma_{A}^{*}$ and geometric models for all marked Riemann
surfaces $(X,h_{X})$ by $X_{0}$. For any finite strings
$w_{n}=i_{0}i_{1}\cdots i_{n-1}i_{n}$ with $a_{i_{l}i_{l+1}}=1$,
$0\leq l\leq n-1$, we rewrite it as $w^{*}_{n} =
j_{n}j_{n-1}\cdots j_{1}j_{0}$, where $j_{n}=i_{0}$,
$j_{n-1}=i_{1}$, $\cdots$, $j_{1}=i_{n-1}$, $j_{0}=i_{n}$. Then
$a_{j_{l}j_{l-1}} =1$. Define the right cylinder
$$
[w_{n}^{*}] = \{ \tilde{w}^{*} = \cdots j_{n+m+1}'\cdots
j_{n+m}'\cdots j_{n+1}'w_{n}^{*}\}
$$
for all $j_{n+m}'\in \{1, \cdots, k\}$ and
$a_{j_{n+m+1}'j_{n+m}'}=1$, $m=1, \cdots$ and
$a_{j_{n+1}'j_{n}}=1$ . All these right cylinders form a
topological basis for the space
$$
\Sigma_{A}^{*} =\{ w^{*} =\cdots j_{n}\cdots j_{1}j_{0}\;|\;
j_{n}\in \{ 1, \cdots, k\}\}.
$$
We call this topological space the dual symbolic space. The left
shift $\sigma_{A}^{*}:\Sigma_{A}^{*}\to \Sigma_{A}^{*}$ is defined
as
$$
\sigma_{A}^{*}: w^{*} =\cdots j_{n}\cdots j_{1}j_{0} \to
\sigma^{*}_{A} (w^{*}) = \cdots j_{n}\cdots j_{1}.
$$

A function
$$
S(w^{*}): \Sigma_{A}^{*}\to {\mathbb R}
$$
is called Lipschitz if there are constants $C>0$ and $0<\mu<1$
such that
$$
|S(w^{*})-S(\tilde{w}^{*})| \leq C\mu^{n}
$$
as long as the first $n$ digits from the right of $w^{*}$ and
$\tilde{w}^{*}$ are the same. The geometric model for a marked
Riemann surface $(X,h_{X})$ by $X_{0}$ is defined as a Lipschitz
function as follows.

Let
$$
f_{X}: S^{1}\to S^{1}
$$
be the transitive expanding Markov map with the Markov partition
$$
\eta_{0} =\{ I_{1}, \cdots, I_{k}\}
$$
for the marked Riemann surface $(X, h_{X})$ by $X_{0}$. Then we
have a sequence of Markov partitions
$$
\eta_{n}=f_{X}^{-n}\eta_{0}
$$
for $n=1, 2, \cdots$. Each interval in
$\eta_{n}$ has a unique labeling $w_{n}^{*}= j_{n}\cdots
j_{1}j_{0}$, which we denote as $I_{w_{n}^{*}}$. Then
$I_{\sigma^{*}_{A} (w_{n}^{*})}$ is an interval in $\eta_{n-1}$,
where $\sigma^{*}_{A} (w_{n}^{*})= j_{n}\cdots j_{1}$. We have
that $I_{w_{n}^{*}}\subset I_{\sigma^{*}_{A} (w_{n}^{*})}$.
Define the pre-scaling function at $w_{n}^{*}$ as
$$
S_{X}(w^{*}_{n}) =\frac{|I_{w_{n}^{*}}|}{|I_{\sigma^{*}_{A}
(w_{n}^{*})}|}.
$$

\begin{lemma}~\label{scalingfunction}
For any $w^{*}= \cdots j_{n}\cdots j_{1}j_{0}$, let $w_{n}^{*}
=j_{n}\cdots j_{1}j_{0}$, then
$$
S_{X}(w^{*})= \lim_{n\to \infty} S_{X}(w^{*}_{n})=\lim_{n\to
\infty} \frac{|I_{w_{n}^{*}}|}{|I_{\sigma^{*}_{A} (w_{n}^{*})}|}
$$
exists and converges uniformly on $w^{*}\in \Sigma_{A}^{*}$.
Moreover,
$$
S_{X}(w^{*}): \Sigma_{A}^{*} \to (0,1)
$$
defines a Lipschitz continuous function.
\end{lemma}

\begin{proof}
For any $w^{*}= \cdots j_{n}\cdots j_{1}j_{0}$, let $w_{n}^{*}
=j_{n}\cdots j_{1}j_{0}$, we consider the sequence $\{
S_{X}(w^{*}_{n})\}_{n=0}^{\infty}$. For any $m>n>0$,
$$
S_{X}(w^{*}_{m}) =\frac{|I_{w_{m}^{*}}|}{|I_{\sigma^{*}_{A}
(w_{m}^{*})}|} =
\frac{|(f^{m-n})'(x)|}{|(f^{m-n})'(y)|}\frac{|I_{w_{n}^{*}}|}{|I_{\sigma^{*}_{A}
(w_{n}^{*})}|}=\frac{|(f^{m-n})'(x)|}{|(f^{m-n})'(y)|}S_{X}(w^{*}_{n})
$$
for some $x\in I_{w_{n}^{*}}$ and $y\in I_{\sigma^{*}_{A}
(w_{n}^{*})}$. Thus we get
$$
|S_{X}(w^{*}_{m}) -S_{X}(w^{*}_{n})|  \leq
\Big|\frac{|(f^{m-n})'(x)|}{|(f^{m-n})'(y)|}-1\Big|
S_{X}(w^{*}_{n})\leq
\Big|\frac{|(f^{m-n})'(x)|}{|(f^{m-n})'(y)|}-1\Big|
$$
From the calculation which we got~(\ref{distortion1}), we have a
constant $C>0$ such that
$$
\Big| \frac{|(f^{m-n})'(x)|}{|(f^{m-n})'(y)|}-1\Big| \leq C
\mu^{n}.
$$
Thus,
$$
|S_{X}(w^{*}_{m}) -S_{X}(w^{*}_{n})| \leq C \mu^{n}.
$$
This implies that $\{ S_{X}(w^{*}_{n})\}_{n=0}^{\infty}$ is a
Cauchy sequence and
$$
S_{X}(w^{*})= \lim_{n\to \infty} S_{X}(w^{*}_{n})=\lim_{n\to
\infty} \frac{|I_{w_{n}^{*}}|}{|I_{\sigma^{*}_{A} (w_{n}^{*})}|}
$$
exists. Furthermore, the limit is uniformly on $\Sigma_{A}^{*}$.

Moreover, if we consider $w^{*}$ and $\tilde{w}^{*}$ with the same
first $n$ digits from the right, then we can write them as
$w^{*}=\cdots w^{*}_{n}$ and $\tilde{w}^{*} =\cdots w_{n}^{*}$. So
we have that for any $m\geq n$,
$$
|S_{X}(w^{*}_{m})-S_{X}(\tilde{w}^{*}_{m})| =
\Big|\frac{|(f^{m-n})'(x)|}{|(f^{m-n})'(y)|}-
\frac{|(f^{m-n})'(\tilde{x})|}{|(f^{m-n})'(\tilde{y})|}\Big|
S_{X}(w^{*}_{n}) \leq 2C\mu^{n}
$$
for some $x,y\in I_{\sigma^{*}_{A}(w_{m}^{*})}$ and
$\tilde{x},\tilde{y}\in I_{\sigma^{*}_{A}(\tilde{w}_{m}^{*})}$. By
taking limit, we get
$$
|S_{X}(w^{*})-S_{X}(\tilde{w}^{*})|\leq 2C\mu^{n}.
$$
Therefore,
$$
S_{X}(w^{*}): \Sigma_{A}^{*} \to (0,1)
$$
is a Lipschitz function.
\end{proof}

\begin{definition} For a marked Riemann surface $(X,h_{X})$ by $X_{0}$,
we call the Lipschitz function
$$
S_{X}(w^{*}): \Sigma_{A}^{*} \to (0,1)
$$
its scaling function.
\end{definition}

The scaling function is a geometric model because of the following
theorem.

\vspace*{10pt}
\begin{theorem}~\label{Ji1}
Suppose $(X,h_{X})$ and $(Y,h_{Y})$ are two marked Riemann
surfaces by $X_{0}$. Then there is a conformal map $\alpha: X\to
Y$ such that $h_{Y}^{-1}\circ \alpha\circ h_{X}: X_{0}\to X_{0}$
is homotopic to the identity if and only if $S_{X} =S_{Y}$.
\end{theorem}

\begin{proof}
We first prove the ``only if" part. Suppose there is a conformal
map $\alpha: X\to Y$ such that $h_{Y}^{-1}\circ \alpha\circ h_{X}:
X_{0}\to X_{0}$ is homotopic to the identity. Let
$$
H_{X}, H_{Y}, \Psi: {\mathbb D}\to {\mathbb D}
$$
be lifts of $h_{X}$, $h_{Y}$, and $\alpha$. Then $H_{Y}^{-1}\circ
\Psi\circ H_{X}: {\mathbb D}\to {\mathbb D}$ is a lift of
$h_{Y}^{-1}\circ \alpha\circ h_{X}: X_{0}\to X_{0}$. From
Theorem~\ref{EarleMcMullen}, there is a lift $H_{Y}^{-1}\circ
\Psi\circ H_{X}$ whose restriction to the unit circle $S^{1}$ is
the identity. Thus
$$
H_{Y} (x) = \Psi \circ H_{X}(x), \quad \forall \; x\in S^{1}.
$$
Note that $h_{X}$ and $h_{Y}$ induce isomorphisms from
$\Gamma_{0}$ to $\Gamma_{X}$ and $\Gamma_{Y}$, respectively, and
$H_{X}$ and $H_{Y}$ are the corresponding boundary
correspondences. From the definition of the Markov maps $f_{X}$
and $f_{Y}$, we have that
$$
\Psi \circ f_{X}= f_{Y}\circ \Psi.
$$
This further implies that $I_{w^{*}_{n},Y} =
\Psi(I_{w^{*}_{n},X})$ for all $w^{*}_{n}$. By the mean value
theorem, we have $\xi_{n}\in I_{w_{n}^{*},Y}$ and $\eta_{n}\in
I_{\sigma^{*}_{A} (w_{n}^{*}),Y}$ such that
$$
S_{Y}(w^{*}) =\lim_{n\to \infty}
\frac{|I_{w_{n}^{*},Y}|}{|I_{\sigma^{*}_{A} (w_{n}^{*}),Y}|} =
\lim_{n\to \infty}
\frac{|\Psi(I_{w_{n}^{*},X})|}{|\Psi(I_{\sigma^{*}_{A}
(w_{n}^{*}),X})|}
$$
$$
= \lim_{n\to \infty}
\frac{|\Psi'(\xi_{n})|}{|\Psi'(\eta_{n})|}\frac{|I_{w_{n}^{*},X}|}
{|I_{\sigma^{*}_{A} (w_{n}^{*}),X}|}
 = \lim_{n\to \infty}
\frac{|I_{w_{n}^{*},X}|}{|I_{\sigma^{*}_{A} (w_{n}^{*}),X}|} =
S_{X}(w^{*})
$$
since $|\xi_{n}-\eta_{n}| \leq |I_{\sigma^{*}_{A} (w_{n}^{*}),X}|
\to 0$.

Now we prove the ``if" part. Consider $h=h_{Y}\circ h_{X}^{-1}:
X\to Y$. Suppose $X={\mathbb D}/\Gamma_{X}$ and $Y={\mathbb
D}/\Gamma_{Y}$. Let $\phi_{XY}: \Gamma_{X}\to \Gamma_{Y}$ be the
isomorphism induced by $h$. Let $H: S^{1}\to S^{1}$ be the
boundary correspondence, that is,
$$
H\circ \gamma (x)= \phi_{XY} (\gamma )\circ H (x)
$$
for all $\gamma\in \Gamma_{X}$ and $x\in S^{1}$.  From the
definition of $f_{X}$ and $f_{Y}$, $H$ is a topological conjugacy
from $f_{X}$ to $f_{Y}$ on $S^{1}$, that is,
$$
f_{Y}\circ H =H\circ f_{X}.
$$
From $S_{Y}=S_{X}$, we claim that $H$ is a Lipschitz map from
$S^{1}$ to $S^{1}$.

We prove this claim. For any interval $I_{w_{n}^{*},X}\in
\eta_{n,X}$, $I_{w_{n}^{*},Y}= H(I_{w_{n}^{*},X}) \in \eta_{n,Y}$.
Then
$$
\Big|\log
\Big(\frac{|H(I_{w_{n}^{*},X})|}{|I_{w_{n}^{*},X}|}\Big)\Big|=
\Big|\log \Big(\frac{|I_{w_{n}^{*},Y}|}{|I_{w_{n}^{*},X}|}
\Big)\Big|=|\log |I_{w_{n}^{*},Y}| -\log |I_{w_{n}^{*},X}||
$$
$$
=\Big|\sum_{k=0}^{n-1} \Big(\log S_{Y}(w^{*}_{n-k})-\log
S_{X}(w^{*}_{n-k})\Big) + \log |I_{j_{0},Y}| - \log
|I_{j_{0},X}|\Big|
$$
$$
\leq \sum_{k=0}^{n-1} \Big| \log S_{Y}(w^{*}_{n-k})-\log
S_{X}(w^{*}_{n-k})\Big| + \Big|\log |I_{j_{0},Y}| - \log
|I_{j_{0},X}|\Big|
$$
$$
\leq \sum_{k=0}^{n-1}
\Big|S_{Y}(w^{*}_{n-k})-S_{X}(w^{*}_{n-k})\Big| + C_{1}
$$
where
$$
C_{1} =\sup_{j_{0}} \Big|\log |I_{j_{0},Y}| - \log
|I_{j_{0},X}|\Big| <\infty.
$$
There are constants $C_{2}>0$ and $0<\mu<1$ such that that
$$
|S_{X}(w^{*}_{n-k}) - S_{X}(w^{*})| \leq C_{2} \mu^{n-k}
$$
and
$$
|S_{Y}(w^{*}_{n-k}) - S_{Y}(w^{*})| \leq C_{2} \mu^{n-k}.
$$
But $S_{X}(w^{*})=S_{Y}(w^{*})$. So we have that
$$
\Big|S_{X}(w^{*}_{n-k})-S_{Y}(w^{*}_{n-k})\Big| \leq 2C_{2}
\mu^{n-k}.
$$
This implies that
$$
C^{-1}_{3} \leq \frac{|H(I_{w_{n}^{*},X})|}{|I_{w_{n}^{*},X}|}
\leq C_{3}
$$
where $C_{3} = \exp (2C_{2}/(1-\mu) +C_{1})$.

Since the set of all endpoints of $\{ H(I_{w_{n}^{*},X})\}$ and
the set of all endpoints $\{I_{w_{n}^{*},X}\}$ are both dense in
$S^{1}$, so the additive formula implies that for any $x,y\in
S^{1}$,
$$
C_{3}^{-1} \leq \frac{|H(x)-H(y)|}{|x-y|}\leq C_{3}.
$$
That is $H$ is a bi-Lipschitz map. Therefore, $H$ is
differentiable almost everywhere and has a point $x_{0}\in S^{1}$
such that $H'(x_{0})\neq 0$. According to the following Tukia's
theorem, $H$ is a M\"obius transformation $\Psi$ on $S^{1}$. This
implies that
$$
H_{Y}^{-1} \circ \Psi \circ H_{X}(x) =x
$$
for all $x\in S^{1}$. From Theorem~\ref{EarleMcMullen},
$$
h_{Y}^{-1}\circ \alpha\circ h_{X}: X_{0}\to X_{0}
$$
is homotopic to the identity, where $\alpha: X\to Y$ is the
conformal map which has a lift $\Psi$.
\end{proof}

The following theorem is first proved by Tukia in~\cite{Tukia}. It
is a stronger version of Mostow's rigidity theorem in the
$2$-dimensional case. We have proved a similar result for
dynamical systems with possibly critical points presented
in~\cite{Jiang1,Jiang2,Jiang3} (see also a survey article~\cite{Jiang8}).
For the completeness of this paper, we give a proof of
Tukia's theorem from the dynamical system point of views.

\vspace*{10pt}
\begin{theorem}[Tukia~\cite{Tukia}]~\label{Tukia}
Suppose $X={\mathbb D}/\Gamma_{X}$ and $Y={\mathbb D}/\Gamma_{Y}$
are two closed hyperbolic Riemann surface of genus $g\geq 2$.
Suppose $\phi: \Gamma_{X}\to \Gamma_{Y}$ is an isomorphism. Let
$H: S^{1}\to S^{1}$ be the boundary correspondence. Then $H$ is a
M\"obius transformation if and only if $H$ is differentiable at
one point with non-zero derivative.
\end{theorem}

\begin{proof}
Since $H$ is the boundary correspondence for $\phi: \Gamma_{X}\to
\Gamma_{Y}$, we have that
$$
H\circ \gamma(x)=\phi(\gamma)\circ H(x)
$$
for any $\gamma\in \Gamma_{X}$ and $x\in S^{1}$.

Consider a marked Riemann surface $(X, h_{X})$ by $X_{0}={\mathbb
D}/\Gamma_{0}$. Then $h_{X}$ induces an isomorphism $\phi_{X}:
\Gamma_{0}\to \Gamma_{X}$. Consider the isomorphism
$\phi_{Y}=\phi\circ \phi_{X}: \Gamma_{0}\to \Gamma_{Y}$ and its
boundary correspondence $H_{Y}: S^{1}\to S^{1}$, that is
\begin{equation}~\label{bc}
H_{Y}\circ \gamma (x) = \phi_{Y} (\gamma) \circ H_{Y}(x)
\end{equation}
for any $\gamma\in \Gamma_{0}$ and $x\in S^{1}$. Then $H_{Y}$ can
be extended to ${\mathbb D}$ still satisfying the above
equation~(\ref{bc}) (this extension may not be unique but the
boundary correspondence is unique). So it induces a quasiconformal
homeomorphism $h_{Y}: X_{0}\to Y$. Thus we get a marked Riemann
surface $(Y, h_{Y})$. Suppose $f_{X}$ and $f_{Y}$ are the
transitive expanding Markov maps corresponding to $(X,h_{X})$ and
$(Y,h_{Y})$. Then
$$
f_{Y} \circ H= H\circ f_X
$$
on $S^{1}$.

If $H$ is a M\"obius transformation, since it is a diffeomorphism
of $S^{1}$,  $H'(x)\neq 0$. This is the ``only if" part.

To prove the ``if" part, suppose $H$ is differentiable at $x_{0}$
with $H'(x_{0})>0$. Let $\{\eta_{n,X}\}_{n=0}^{\infty}$ be the
sequence of Markov partitions for $f_{X}$. Then there is a
sequence of nested intervals $I_{w_{n}}\in \eta_{n,X}$ such that
$$
x_{0}\in \cdots \subset I_{w_{n}}\subset I_{w_{n-1}}\subset \cdots
I_{w_{1}}.
$$
Without loss of generality, we assume that $x_{0}$ is an interior
point of $I_{w_{n}}$ for all $n\geq 0$.

Suppose $H$ is differentiable at a point $x_0$ on the circle. Then
$$
H(x) = H(x_0 ) + H' (x_0 )(x - x_0 ) + o(|x- x_0 |)
$$
for $x$ close to $x_0$.

Consider $\{x_n = f_{X}^n(x_0)\}_{n=0}^{\infty}.$
Let $0 < a < 1$ be a real number. Consider
the interval $I_n = (x_n , x_n + a).$ Let $J_n = (x_0 , z_n )$ be an interval such
that
$$
f_{X}^n : J_n \rightarrow I_n$$
is a $C^{2}$ diffeomorphism.
Let $f_{X}^{-n} : I_n  \rightarrow J_n$ denote its inverse.
Since $f_{X}$ is expanding, the length $|J_n| \rightarrow 0$ as $n \rightarrow \infty.$
Similarly, we have
$$
f_{Y}^n : H(J_n ) \rightarrow  H(I_n )
$$
is a $C^{2}$ diffeomorphism. Let $f_{Y}^{-n} : H(I_n ) \rightarrow H(J_n )$ be its inverse. Then
$$
H(x) = f_{Y}^n \circ H \circ f_{X}^{-n}(x),  \ \ \ x \in I_n.$$
Let
$$
\alpha_n (x) = \frac{x-x_0}{x_n-x_0}: J_n
\rightarrow (0, 1)
$$
and
$$
\beta_n (x) = \frac{x -H(x_0 )}{H(x_n )-H(x_0)}: H(J_n )
\rightarrow (0, 1).
$$
Then
$$
H(x) = (f_{Y}^n \circ \beta_n^{-1}) \circ (\beta_n \circ  H \circ \alpha_n^{-1})
\circ (\alpha_n \circ f_{X}^{-n})(x), \;\;\;x \in  I_n.
$$
The key estimate comes from the following distortion result (the proof is similar to the proof of Lemma~\ref{boundedgeometry} or refer to~\cite[Chapter 1]{Jiang6}):
There is a constant $C > 0$ independent of $n$ and any inverse branches of $f_{X}^n$ and $f_{Y}^n$ such that
\begin{equation}~\label{distortion2}
\left| \log |\frac{(f_{X}^{-n})'(x))}{(f_{X}^{-n})'(y))}\ |\right|
\leq
C |x-y|,  {\rm \ for \ all \ }  x {\rm \ and \ }  y {\rm \ in \ }  I_n
\end{equation}
and
\begin{equation}~\label{distortion3}
\left| \log |\frac{(f_{Y}^{-n})'(x))}{(f_{Y}^{-n})'(y))}\ |\right|
\leq
C |x-y|,  {\rm \ for \ all \ }  x {\rm \ and \ }  y {\rm \ in \ }  H(I_n).
\end{equation}

From this distortion property, one can conclude that $f_{Y}^n
\circ \beta_n^{-1}$ and
$\alpha_n  \circ f_{X}^{-n}$  are sequences of bi-Lipschitz
homeomorphisms with a uniform
Lipschitz constant. Therefore, they have convergent subsequences. Without loss of generality, let us assume that these two sequences themselves are convergent. The map $\beta_n \circ  H \circ \alpha_n^{-1}$ converges to a linear map.

Since the unit circle is compact and all $I_n$ have fixed length $a$, there is a subsequence $I_{n_{i}}$ of intervals such that $\cap_{i=1}^{\infty} I_{n_{i}}$ contains an interval
$I$ of positive length. Without loss of generality, let us assume that $\cap_{n=1}^{\infty} I_n$ contains an interval
$I$ of positive length. Thus $H$ is a bi-Lipschitz homeomorphism on $I$.

Since $H|I$ is bi-Lipschitz, $H'$ exists a.e. in $I$ and is
integrable. Since $(H|I)'(x)$ is measurable and $H|I$ is a
homeomorphism, we can find a point $y_{0}$ in $I$ and a subset
$E_{0}$ containing $y_{0}$ such that
\begin{itemize}
\item[{\rm 1)}] $H|I$ is differentiable at every point in $E_{0}$;
\item[{\rm 2)}] $y_{0}$ is a density point of $E_{0}$;
\item[{\rm 3)}] $H'(y_{0})\neq 0$; and
\item[{\rm 4)}] the derivative $H'|E_{0}$ is
continuous at $y_{0}$.
\item[{\rm 5)}] $y_{0}$ is not an endpoint of an interval in
$\eta_{n,X}$ for all $n\geq 0$.
\end{itemize}

Since $S^{1}$ is compact, there is a subsequence $\{f_{X}^{
n_{k}}(y_{0})\}_{k=1}^{\infty}$ converging to a point $z_{0}$ in
$S^{1}$. Let $I_{0}=(b,c)$ be an open interval such that
$z_{0}\in \overline{I}_{0}$. There is a sequence of interval
$\{I_{k}\}_{k=1}^{\infty}$ such that $y_{0}\in
 I_{k}\subseteq I$ and $f_{X}^{ n_{k}}: I_{k}\rightarrow I_{0}$ is a
$C^{2}$ diffeomorphism. Then $|I_{k}|$ goes to zero as $k$ tends
to infinity.

From the distortion property~(\ref{distortion1}), there is a
constant $C_{4}>0$, such that
$$
\Big| \log \Big( \frac{|(f_{X}^{ n_{k}})'(w)|}{|(f_{X}^{
 n_{k}})'(z)|}\Big) \Big|\leq  C_{4}, \quad \forall w, z\in I_{k},\;\; \forall  k\geq 1.
$$

Since $y_{0}$ is a density point of $E_{0}$, for any integer
$s>0$, there is an integer $k_{s}>0$ such that
$$ \frac{|E_{0}\cap
I_{k}|}{|I_{k}|} \geq 1-\frac{1}{s}, \quad \forall k\geq k_{s}.
$$
Let $E_{k}=f_{X}^{n_{k}}(E_{0}\cap I_{k})$. Then $H$ is
differentiable at every point in $E_{k}$ and, from the distortion
property~(\ref{distortion1}), there is a constant $C_{5}>0$ such
that
$$
\frac{|E_{k}\cap I_{0}|}{|I_{0}|} \geq 1-\frac{C_{5}}{s}, \quad
\forall k\geq k_{s}.
$$
Let
$$
E=\cap_{s=1}^{\infty}\cup_{k\geq k_{s}} E_{k}.
$$
Then $E$ has full measure in $I_{0}$ and $H$ is differentiable at
every point in $E$ with non-zero derivative.

Next, we are going to prove that $H'|E$ is uniformly continuous.
For any $x$ and $y$ in $E$, let $z_{k}$ and $w_{k}$ be the
preimages of $x$ and $y$ under the diffeomorphism
$f_{X}^{n_{k}}:I_{k}\rightarrow I_{0}$. Then $z_{k}$ and $w_{k}$
are in $E_{0}$. From $ H\circ f_{X}=f_{Y}\circ H$, we have that
$$
H'(x) = \frac{(f_{Y}^{ n_{k}})'(H(z_{k}))}{(f_{X}^{
n_{k}})'(z_{k})}H'(z_{k})
$$
and
$$
H'(y) = \frac{(f_{Y}^{
n_{k}})'(H(w_{k}))}{(f_{X}^{n_{k}})'(w_{k})}H'(w_{k}).
$$
So
$$
\Big| \log \Big( \frac{H'(x)}{H'(y)}\Big) \Big| \leq \Big| \log
\Big| \frac{(f_{Y}^{ n_{k}})'(H(z_{k}))}{ (f_{Y}^{
n_{k}})'(H(w_{k}))}\Big| \Big| + \Big| \log \Big| \frac{(f_{X}^{
n_{k}})'(w_{k})}{(f_{X}^{n_{k}})'(z_{k})}\Big| \Big| + \Big| \log
\Big( \frac{H'(z_{k})}{H'(w_{k})}\Big) \Big|.
$$

Since $f_{X}$ and $f_{Y}$ are both piecewise $C^{2}$. From the
distortion property~(\ref{distortion2}) and~(\ref{distortion3}), there is a constant
$C_{6}>0$ such that
$$\Bigl|
\log \Big| \frac{(f_{X}^{ n_{k}})'(w_{k})}{(f_{X}^{
n_{k}})'(z_{k})}\Big| \Bigr| \leq C_{6} |x-y|
$$
and
$$
\Bigl| \log \Big| \frac{(f_{Y}^{
n_{k}})'(H(z_{k}))}{(f_{Y}^{n_{k}})'(H(w_{k}))}\Big| \Bigr| \leq
C_{6} |H(x)-H(y)| $$ for all $k\geq 1$. Therefore,
$$
\Big| \log \Big( \frac{H'(x)}{H'(y)} \Big) \Big| \leq
C_{6}\Big(|x-y| +|H(x)- H(y)|\Big) +\Big| \log \Big(
 \frac{H'(z_{k})}{H'(w_{k})} \Big) \Big|
$$
for all $k\geq 1$. Since $H'|E_{0}$ is continuous at $y_{0}$, the
last term in the last inequality tends to zero as $k$ goes to
infinity. Hence
$$ \Big| \log \Big( \frac{H'(x)}{H'(y)} \Big)
\Big| \leq C_{6}\Big( |x-y| +|H(x)- H(y)|\Big).
$$
This means that $H'|E$ is uniformly continuous. So it can be
extended to a continuous function $\phi$ on $I_{0}$. Because
$H|I_{0}$ is absolutely continuous and $E$ has full measure,
$$ H(x)
=H(a)+\int_{a}^{x}H'(x)dx =H(a) +\int_{a}^{x}\phi(x)dx
$$
on $I_{0}$. This implies that $H|I_{0}$ is actually $C^{1}$.
(This, furthermore, implies that $H|I_{0}$ is $C^{2}$).

Now for any $x\in S^{1}$, let $J$ be an open interval about $x$.
By the expansive and transitivity properties of $f_X$, there is an
integer $n>0$ and an open interval $J_{0}\subset I_{0}$ such that
$f_{X}^{ n}: J_{0}\to J$ is a $C^{1}$ diffeomorphism. By the
equation $H\circ f_{X}=f_{Y}\circ H$, we have that $H|J$ is
$C^{1}$. Therefore, $H$ is $C^{1}$.

Since $H: S^{1}\to S^{1}$ is the boundary correspondence for
$\phi_{XY}: \Gamma_{X}\to \Gamma_{Y}$, we have that
$$
H\circ \gamma (x) =\phi_{XY} (\gamma)\circ H (x)
$$
for all $\gamma\in \Gamma_{X}$ and $x\in S^{1}$. By composition
and post-composition M\"obius transformations, we can assume that
$\gamma (x) =\lambda x$ and $\phi_{XY}(\gamma) (x) =\lambda x$
($C^{1}$-diffeomorphism preserves the eigenvalue at a periodic
point), then we can get that $H(x) =ax$. Thus $H$ is a M\"obius
transformation. We proved the theorem.
\end{proof}

\section{\bf Teichm\"uller space represented by the space of functions}

Let $T(X_{0})$ be the Teichm\"uller space of $X_{0}$. It is the
space of all equivalence classes $\tau=[(X,h_{X})]$ of all marked
Riemann surfaces $(X, h_{X})$ by $X_{0}$. Let
$$
{\mathcal F}=\{ S_{X}\}
$$
be the space of all scaling functions. The following result is a
consequence of Theorem~\ref{Ji1} now.

\vspace*{10pt}
\begin{theorem}~\label{samesf}
Any marked Riemann surfaces $(X,h_{X})$ and $(Y, h_{Y})$ by
$X_{0}$ are in a same point $\tau\in T(X_{0})$ if and only if they
have the same scaling functions, that is, $S_{X}=S_{Y}$.
\end{theorem}

\begin{proof}
Since $(X,h_{X}), (Y, h_{Y})\in \tau$ if and only if there is a
conformal map $\alpha: X\to Y$ such that $h= h_{Y}^{-1}\circ
\alpha \circ h_{X}: X_{0}\to X_{0}$ is homotopic to identity.
\end{proof}

Thus we can denote $S_{\tau}=S_{X}$ for any $(X,h_{X})\in \tau\in
T(X_{0})$ and introduce a bijective map from the Teichm\"uller
space $T(X_{0})$ to the function space ${\mathcal F}$,
$$
\iota : T(X_{0})\to {\mathcal F}; \quad \iota (\tau) =S_{\tau},
\quad \tau\in T(X_{0}).
$$

Now suppose $R$ is any closed Riemann surface of genus $g$. Let
$h_{0}: X_{0}\to R$ be a quasiconformal homeomorphism. For any
marked Riemann surface $(X,h_{X})$ by $R$, we have a marked
Riemann surface $(X, h_{X}\circ h_{0})$ by $X_{0}$. For any two
marked Riemann surfaces $(X,h_{X})$ and $(Y,h_{Y})$ by $R$, if
there is a conformal map $\alpha: X\to Y$ such that
$$
h_{Y}^{-1}\circ \alpha\circ h_{X}: R\to R
$$
is homotopic to the identity, then
$$
h_{0}^{-1}\circ h_{Y}^{-1}\circ \alpha\circ
h_{X}\circ h_{0}: X_{0}\to X_{0}
$$
is also homotopic to the identity. Thus this gives a bijective map
$\vartheta: T(R)\to T(X_{0})$. Therefore, we have a bijective map
$$
\iota_{R}= \iota\circ \vartheta: T(R)\to {\mathcal F}.
$$

\section{\bf Bers' embedding for ${\mathcal F}$.}

Let $R={\mathbb D}/\Gamma$ be a closed hyperbolic Riemann surface
of genus $g\geq 2$. The group $\Gamma$ acts on the whole Riemann
sphere ${\mathbb P}^{1}$. The limit set of $\Gamma$ is just the
unit disk $S^{1}$. The quotient
$$
R^{*} = ({\mathbb P}^{1}\setminus \overline{{\mathbb D}})/\Gamma
$$
of the outer of the closed unit disk is another closed hyperbolic
Riemann surface complex conjugate to $R$. Let $T(R)$ be the
Teichm\"uller space of $R$. Suppose $\tau=[(X,h_{X})]\in T(R)$ and
$X={\mathbb D}/\Gamma_{X}$. Then $h_X$ can be lift to a
quasiconformal homeomorphism $H: {\mathbb D}\to {\mathbb D}$
conjugating the Fuchsian group $\Gamma$ and $\Gamma_{X}$. Let
$\mu_{H}=H_{\overline{z}}/H_{z}$ be the Belrtami coefficient of
$H$. Extend it to the whole Riemann sphere ${\mathbb P}^{1}$ by
$$
\mu (z)   =\left\{
\begin{array}{ll}
       \mu_{H}(z), & z\in {\mathbb D}\cr
       0,          & z\in {\mathbb D}^{*}= {\mathbb P}^{1}\setminus \overline{\mathbb
       D}.
\end{array}\right.
$$
Let $\Phi(z): {\mathbb P}^{1}\to {\mathbb P}^{1}$ be the
normalized solution to the Beltrami equation
$$
\Phi_{\overline{z}} =\mu (z) \Phi_{z}.
$$
Since $\mu$ is $\Gamma$-invariant, we have that
$$
\gamma_{*}(\mu_{\Phi}) (z) =\mu_{\Phi}(z)
$$
for all $\gamma\in \Gamma$. Since the solution of the Beltrami
equation is unique up to post-composition with a M\"obius
transformation, so there is an isomorphism
$$
\phi: \Gamma\to \Gamma_{X}
$$
such that
$$
\Phi \circ \gamma =\phi(\gamma)\circ \Phi
$$
for all $\gamma\in \Gamma$. Thus $\Phi$ conjugates $\Gamma$ to
$\Gamma_{X}$. By the construction of $\Phi$, it is conformal in
${\mathbb D}^{*}$. This implies that
$$
R^{*} = \Phi ({\mathbb D}^{*})/\Gamma_{X} \quad \hbox{and}\quad
X=\Phi ({\mathbb D})/\Gamma_{X}.
$$
The Schwarzian derivative
$$
s(\Phi)(z) = \Big(\frac{\Phi'''(z)}{\Phi'(z)} -\frac{3}{2}
\Big(\frac{\Phi''(z)}{\Phi'(z)}\Big)^{2}\Big) dz^{2}, \quad z\in
{\mathbb D}^{*},
$$
is $\Gamma$-invariant. It induces a quadratic differential on
$R^{*}$ and is independent of the choice of $(X,h_{X})\in \tau$.
Thus we get a quadratic differential $s(\tau)$ on $R^{*}$. Let
$Q(R^{*})$ be the space of all quadratic differentials $q=q(z)
dz^{2}$ on $R^{*}$ with the norm
$$
||q|| = \sup_{z\in R^{*}} (|q(z)| \rho^{-2}(z)).
$$
where $\rho$ is the hyperbolic metric on $S^{*}$. Then it is a
complex dimension $3g-3$ linear space. Then we can embed $T(R)$
into $Q(R^{*})$ by
$$
s(\tau): T(R) \to Q(R^{*}).
$$
It is called Bers's embedding. So we can think $T(R)$ as an open
domain of $Q(R^{*})$.

Now take $R=X_{0}$ and consider
$$
s\circ \iota^{-1}: {\mathcal F}\to Q(R^{*}).
$$
Then we can embed ${\mathcal F}$ into a complex manifold
$Q(R^{*})$. Let $B(1/2)$ and $B(3/2)$ be the balls of radii $1/2$
and $3/2$ in $Q(R^{*})$. Then we have that
$$
B(1/2)\subset s(\iota^{-1}({\mathcal F})) \subset B(3/2).
$$

\section{\bf Teichm\"uller metric and maximum metric.}

The Teichm\"uller metric on $T(R)$ is defined as follows. Given
any two points $\tau, \tau'\in T(R)$. Let $(X,h_{X})\in \tau$ and
$(Y, h_{Y})\in \tau'$. Then $h_{XY}=h_{Y}\circ h_{X}^{-1}: X\to Y$
is a quasiconformal homeomorphism. Define
$$
d_{T} (\tau, \tau') =\inf \frac{1}{2} \inf \{ \log K(h)\;|\; h:
X\to Y \hbox{ is homotopic to $h_{XY}$}\}
$$
where the first $inf$ takes over all marked Riemann surfaces
$(X,h_{X})\in \tau$ and $(Y,h_{Y})\in \tau'$ by $R$ and where
$$
K(h) =\sup_{z\in X}
\frac{|h_{z}|+|h_{\overline{z}}|}{|h_{z}|-|h_{\overline{z}}|}.
$$
Since the map $\vartheta: T(R)\to T(X_{0})$ is an isometry, we
only need to consider the case $R=X_{0}$ for the purpose of the
study of metrics.

Suppose $R=X_{0}$. Since $\iota: T(R)\to {\mathcal F}$ is one-to-one and onto, we
can define the Teichm\"uller metric on ${\mathcal F}$ as
$$
d_{T}(S, S') = d_{T}(\iota^{-1}(S), \iota^{-1}(S')).
$$

Since $S=\iota(\tau)$ is a function on $\Sigma_{A}^{*}$, it has a
natural maximum norm
$$
||S|| =\sup_{w^{*}\in \Sigma_{A}^{*}} |S(w^{*})|.
$$
The distance is defined as
$$
d(S, S') =||S-S'||.
$$
This introduces a metric $d_{max}(\cdot, \cdot)$ on the
Teichm\"uller space $T(R)$,
$$
d_{max} (\tau, \tau') = d (\iota (\tau), \iota (\tau'))
$$
for any $\tau, \tau'\in T(R)$.

\vspace*{10pt}
\begin{theorem}~\label{finermetric}
The identity map
$$
id_{TM}: (T(R), d_{T}) \to (T(R), d_{max})
$$
is uniformly continuous.
\end{theorem}

To prove this theorem, we need several lemmas. Weierstrass'
${\mathcal P}$-function is defined by
$$
{\mathcal P} (z) = \frac{1}{z^{2}} +\sum \Big(
\frac{1}{(z-m\omega_{1}-n\omega_{2})^{2}} -
\frac{1}{(m\omega_{1}+n\omega_{2})^{2}}\Big).
$$
It is a solution of the differential equation
$$
{\mathcal P}'(z)^{2} = 4 ({\mathcal P}(z)-e_{1}) ({\mathcal
P}(z)-e_{2}) ({\mathcal P}(z)-e_{3})
$$
where
$$
e_{1}= {\mathcal P}(\frac{\omega_{1}}{2}),\;\; e_{2}= {\mathcal
P}(\frac{\omega_{2}}{2}),\;\; e_{3}= {\mathcal
P}(\frac{\omega_{1}+\omega_{2}}{2}).
$$
They are distinct numbers.

Let $\kappa=\omega_{2}/\omega_{1}$ and consider only the
half-plane $\Im \kappa >0$. Then we have a function
$$
\rho (\kappa) = \frac{e_{3}-e_{1}}{e_{2}-e_{1}}.
$$

\vspace*{10pt}
\begin{lemma}~\label{halfvalue} $\rho (\kappa)\not= 0,1$ is an analytic function and
$$
\rho(i)= \frac{1}{2}.
$$
\end{lemma}

Suppose ${\mathbb H}$ is the upper half plane. It is another model
of the hyperbolic disk ${\mathbb D}$. Suppose $\Phi: {\mathbb
H}\to {\mathbb H}$ is a $K$-quasiconformal orientation-preserving
homeomorphism. Then it can be extended to a homeomorphism, which
we still denote as $\Phi$, of ${\mathbb H}\cup {\mathbb R}$. Let
$\psi: {\mathbb R}\to {\mathbb R}$ be the restriction of $\Phi$ to
the boundary of ${\mathbb H}$. Then $\psi$ is a quasisymmetric
homeomorphism and $\psi (-\infty)=-\infty$ and
$\phi(+\infty)=+\infty$. We have that

\vspace*{10pt}
\begin{lemma}~\label{upper}
$$
\lambda (K)^{-1}=\frac{1-\rho (iK)}{\rho(iK)} \leq \frac{|\psi
(x+t)-\psi(x)|}{|\psi (x)-\psi(x-t)|}\leq \lambda
(K)=\frac{\rho(iK)}{1-\rho (iK)}.
$$
\end{lemma}

The proofs of the above two lemmas can be found in Ahlfors'
book~\cite{Ahlfors}. The further estimation of $\lambda (K)$ is
that (refer to~\cite{Lehto})
\begin{equation}~\label{lambda}
\lambda (K) \leq e^{5(K-1)}.
\end{equation}

Let $\varphi: S^{1}\to S^{1}$ be an orientation preserving
homeomorphism. Suppose $\epsilon >0$ and $M>1$ are two constants.
We call it $(\epsilon, M)$-quasisymmetric if
$$
M^{-1} \leq \frac{|\varphi(x) -
\varphi(\frac{x+y}{2})|}{|\varphi(\frac{x+y}{2})-\varphi(y)|}\leq
M
$$
for any $x, y\in S^{1}$ and $|x-y|\leq \epsilon$, where $|\cdot|$
means the Lebesgue metric on $S^{1}$.

By considering the equality~(\ref{lambda}) or
Lemmas~\ref{halfvalue} and~\ref{upper} and $\rho (\kappa)$ is
continuous at $i$, we have that

\vspace*{10pt}
\begin{lemma}~\label{MK}
There are two bounded functions $\epsilon (K)>0$ and $M (K)> 1$
with $\epsilon (K) \to 0^{+}$ and $M (K) \to 1^{+}$ as $K\to
1^{+}$ such that if $H$ is a $K$-quasiconformal homeomorphism of
${\mathbb D}$, then $\varphi=H|S^{1}$ is a $(\epsilon(K),
M(K))$-quasisymmetric homeomorphism of $S^{1}$.
\end{lemma}

\begin{proof}
Suppose $\Upsilon$ is a M\"obius transformation mapping ${\mathbb
R}$ to $S^{1}$. Then $\Upsilon^{-1} \circ H\circ \Upsilon$ is a
$K$-quasiconformal homeomorphism of ${\mathbb H}$. From
Lemma~\ref{upper},
$$
\frac{|\Upsilon'(\xi)|}{|\Upsilon'(\eta)|} \lambda (K)^{-1}\leq
\frac{|\varphi(x) -
\varphi(\frac{x+y}{2})|}{|\varphi(\frac{x+y}{2})-\varphi(y)|}\leq
\frac{|\Upsilon'(\xi)|}{|\Upsilon'(\eta)|} \lambda (K)
$$
where $\xi, \eta\in [x,y]$. Without loss of generality, we assume
that $\Upsilon^{-1}([x,y])$ is in a fixed compact set of ${\mathbb
R}$ (otherwise, we use a different $\Upsilon$ such that
$\Upsilon^{-1}([x,y])$ away from $\infty$). Thus we have a number
$\epsilon (K)>0$ such that
$$
\frac{|\Upsilon'(\xi)|}{|\Upsilon'(\eta)|} \leq e^{5(K-1)}
$$
for any $|y-x| \leq \epsilon (K)$. Thus we can take
$$
M(K) =  e^{10(K-1)}\to 1^{+}, \quad \hbox{as}\quad K\to 1^{+}.
$$
\end{proof}

Suppose $[a,b]$ is an interval and $H: [a,b]\to H([a,b])$ is a
homeomorphism. We say $H$ is $M$-quasisymmetric on $[a,b]$ if
$$
M^{-1} \leq \frac{|H(x+t)-H(x)|}{|H(x)-H(x-t)|} \leq M
$$ for any
$x, x+t,x-t\in [a,b]$ and $t>0$. We give a proof of the following
lemma.

\vspace*{10pt}
\begin{lemma}~\label{sd}
There is a bounded function $\zeta (M)>0$ satisfying $\zeta (M)
\to 0$ as $M\to 1^{+}$ such that for any $M$-quasisymmetric
homeomorphism $H$ of $[0,1]$ with $H(0)=0$ and $H(1)=1$,
$$
|H(x)-x|\leq \zeta (M), \quad \forall\; x \in [0,1].
$$
\end{lemma}

\begin{proof}
Consider points $x_{n}=1/2^{n}$, $n=0, 1, \cdots$. The
$M$-quasisymmetry condition implies that
$$
M^{-1}\leq
\frac{H(\frac{1}{2^{n-1}})-H(\frac{1}{2^{n}})}{H(\frac{1}{2^{n}})-H(0)}\leq
M.
$$
From this and the fact that $H(0)=0$, we get
$$
(1+M^{-1})H(\frac{1}{2^{n}})\leq H(\frac{1}{2^{n-1}}) \leq
(1+M)H(\frac{1}{2^{n}}).
$$
This gives
$$
\frac{1}{1+M} H(\frac{1}{2^{n-1}}) \leq H(\frac{1}{2^{n}}) \leq
\frac{1}{1+M^{-1}} H(\frac{1}{2^{n-1}}).
$$
Using the fact that $H(1)=1$, we further get
$$
\Big(\frac{1}{1+M}\Big)^{n} \leq H(\frac{1}{2^{n}}) \leq
\Big(\frac{1}{1+M^{-1}}\Big)^{n}, \quad \forall\; n\geq 1.
$$
Furthermore, by $M$-quasisymmetry  and induction on $n=1,
2,\cdots$, yield
$$
\Big(\frac{1}{1+M}\Big)^{n} \leq H(\frac{i}{2^{n}}) -
H(\frac{i-1}{2^{n}}) \leq \Big( \frac{1}{1+M^{-1}}\Big)^{n}, \quad
\forall \; n\geq 1, \;\; 1\leq i\leq 2^{n}.
$$

Let
$$
\chi_n=
\max\left\{\left(\frac{M}{M+1}\right)^n-\frac{1}{2^n},\frac{1}{2^n}-\left(\frac{1}{M+1}\right)^n\right\},
\quad n=1,2, \cdots.
$$
Then for $n=1$,
$$
|H(\frac{1}{2}) -\frac{1}{2}| \leq
\chi_{1}=\frac{1}{2}\frac{M-1}{M+1},
$$
and for any $n>1$, we have
$$
\max_{0\leq i\leq 2^{n}} \Big| H(\frac{i}{2^{n}})
-\frac{i}{2^{n}}\Big| \leq \max_{0\leq i\leq 2^{n-1}} \Big|
H(\frac{i}{2^{n-1}}) -\frac{i}{2^{n-1}}\Big| + \chi_{n}
$$
By summing over $k$ for $1 \leq k \leq n,$ we obtain
$$
\max_{0\leq i\leq 2^{n}} \Big| H(\frac{i}{2^{n}})
-\frac{i}{2^{n}}\Big| \leq \delta_{n}=\sum_{k=1}^{n} \chi_{k}.
$$
If we put $\zeta (M) = \sup_{1\leq n<\infty}\{\delta_{n}\},$ by
summing geometric series, we obtain
$$
\zeta(M) = \max_{1\leq n<\infty} \Big\{
M-1+\frac{1}{2^{n}}-M\Big(\frac{M}{1+M}\Big)^{n},
1-\frac{1}{M}+\frac{1}{M}\Big(\frac{1}{M}\Big)^{n}
-\frac{1}{2^{n}}\Big\}.
$$
Clearly,  $\zeta(M)\to 0$ as $M\to 1$, and since the dyadic points
$$
\{ i/2^{n}\;\; |\;\; n=1, 2, \cdots ; 0\leq i\leq 2^{n}\}
$$
are dense in $[0,1]$, we conclude
$$
|H(x)-x| \leq \zeta (M) \quad \forall \; x\in [0,1],
$$
which proves the lemma.
\end{proof}

Concluding from the above four lemmas, we have that

\vspace*{10pt}
\begin{lemma}~\label{DTtoD}
There is a bounded function $\varrho(\xi) >0$ with $\varrho
(\xi)\to 0$ as $\xi\to 0$ such that
$$
d_{max} (\tau, \tau') \leq \varrho (d_{T}(\tau, \tau'))
$$
for any two $\tau, \tau'\in T(X_{0})$.
\end{lemma}

\begin{proof}
Suppose $K=\exp(2d_{T}(\tau,\tau'))\geq 1$. Then we have two marked Riemann
surfaces $(X, h_{X})\in \tau$ and $(Y, h_{Y})\in \tau'$ such that
$$
h_{XY}=h_{Y}\circ h_{X}^{-1}: X={\mathbb D}/\Gamma_{X}\to
Y={\mathbb D}/\Gamma_{Y}
$$
is a $K$-quasisconformal homeomorphism. (We can pick $h_{XY}$ as
the Teichm\"uller map.) Then $h_{XY}$ can be lift to a
$K$-quasiconformal homeomorphism $H$ of ${\mathbb D}$ such that
$H|S^{1}$ is the boundary correspondence for the isomorphism from
$\Gamma_{X}\to \Gamma_{Y}$ induced by $h_{XY}$. We still use $H$
to denote its restriction to $S^{1}$. Then, from Lemma~\ref{sd},
it is $(\epsilon(K), M(K))$-quasisymmetric on $S^{1}$ and the
conjugacy between the transitive expanding Markov maps $f_{X}$ and
$f_{Y}$, that is,
$$
H\circ f_{X}= f_{Y}\circ H.
$$

For any point $w^{*}= \cdots w_{n}^{*}\in \Sigma_{A}^{*}$, we have
that
$$
I_{w^{*}_{n}, X}\in \eta_{n,X} \quad \hbox{and}\quad
I_{\sigma^{*}(w^{*}_{n}), X}\in \eta_{n-1,X}
$$
and
$$
I_{w_{n}^{*}, Y}=H(I_{w^{*}_{n}, X})\in \eta_{n,Y}  \quad
\hbox{and}\quad I_{\sigma^{*}(w^{*}_{n}),
Y}=H(I_{\sigma^{*}(w^{*}_{n}), X})\in \eta_{n-1,Y}.
$$
Note that
$$
I_{w^{*}_{n}, X}\subset I_{\sigma^{*}(w^{*}_{n}), X}\quad
\hbox{and}\quad I_{w_{n}^{*}, Y}\subset I_{\sigma^{*}(w^{*}_{n}),
Y}.
$$

Let $n_{0}>0$ be an integer such that
$$
|I_{\sigma^{*}(w^{*}_{n}), X}|\leq \epsilon (K)
$$
for all $n\geq n_{0}$. Then $H|I_{\sigma^{*}(w^{*}_{n}), X}$ is a
$M(K)$-quasisymmetric homeomorphism.

By considering $[0,1]$ gluing $0$ and $1$ as a model of $S^{1}$,
then by rescaling $I_{\sigma^{*}(w^{*}_{n}), X}$ and
$I_{\sigma^{*}(w^{*}_{n}), Y}$ into the unit interval $[0,1]$ by
linear maps, we can think $H|I_{\sigma^{*}(w^{*}_{n}), X}$ is a
$M(K)$-quasisymmetric homeomorphism of $[0,1]$ and fixes $0$ and
$1$. Then Lemma~\ref{sd} implies that
$$
|S_{Y}(w_{n}^{*}) - S_{X}(w_{n}^{*})| =\Big|
\frac{|H(I_{w^{*}_{n}, X})|}{|H(I_{\sigma^{*}(w^{*}_{n}), X})|}
-\frac{|I_{w^{*}_{n}, X}|}{|I_{\sigma^{*}(w^{*}_{n}), X}|}\Big|
\leq \zeta(M(K)) .
$$
This implies that
$$
|S_{Y}(w^{*}) - S_{X}(w^{*})| \leq \zeta(M(K)).
$$
Therefore,
$$
d_{max} (\tau, \tau') \leq \zeta (M (d_{T}(\tau, \tau')).
$$
We take $\varrho (\xi) =\zeta (M(\xi))$. The bounded function
$\varrho (\xi)\to 0$ as $\xi\to 0^{+}$. We completed the proof.
\end{proof}

\begin{proof}[Proof of Theorem~\ref{finermetric}]
For any $\epsilon>0$, there is a $\delta>0$ such that $\varrho
(\xi)<\epsilon$ for any $0\leq \xi<\delta$. Thus for any
$\tau,\tau'\in T(R)$ with $d_{T}(\tau,\tau')<\delta$, from
Lemma~\ref{DTtoD}, $d_{max}(\tau,\tau') \leq \varrho
(d_{T}(\tau,\tau')) <\epsilon$. Thus
$$
id: (T(R), d_{T}(\cdot, \cdot) )\to (T(R),
d_{max}(\cdot, \cdot))
$$
is uniformly continuous. We have proved the theorem.
\end{proof}

\vspace*{10pt}
\begin{theorem}~\label{semifinermetric}
The identity map
$$
id_{MT}: (T(R), d_{max})\to (T(R),
d_{T})
$$
is continuous.
\end{theorem}

\begin{proof} Suppose $id_{MT}: (T(R), d_{max})\to (T(R), d_{T})$ is not continuous.
That is, we have a real number $\epsilon>0$
and a point $S=\iota (\tau)$ and a sequence of points $\{ S_{m}=\iota (\tau_{m}) \}_{m=1}^{\infty}$
in the Teichm\"uller space $T(R)$ such that
$$
d_{\max} (\tau_{m}, \tau)=\|S_{m}-S\| \to 0\quad \hbox{as}\quad m\to \infty
$$
but
$$
d_{T} (\tau_{m}, \tau)\geq \epsilon, \quad \forall\; m.
$$
Let $(X, h_{X})\in \tau$ be a fixed representation
and $(X_{m}, h_{X_{m}})\in \tau_{m}$ for each $m$ be
a representation such that
$$
h_{m}=h_{X_{m}}\circ h_{X}^{-1}: X={\mathbb D}/\Gamma_{X}\to X_{m} ={\mathbb D}/\Gamma_{m}
$$
is a $K_{m}=\exp (2d_{T} (\tau_{m}, \tau))$-quasiconformal homeomorphism.
(We can pick $h_{m}$ as the Teichm\"uller map.)
Then $h_{m}$ can be lift to a $K_{m}$-quasiconformal map $H_{m}$ of ${\mathbb D}$
such that $H_{m}|S^{1}$ is the boundary correspondence for the isomorphism from
$\Gamma_{m}\to \Gamma$ induced by $h_{m}$. We still use $H_{m}$ to denote this
boundary correspondence. Let $f_{X_{m}}$ and $f_{X}$ are the corresponding Markov maps.
Let $\{\eta_{n}\}_{n=0}^{\infty}$ and $\{\eta_{m,n}\}_{n=0}^{\infty}$ be the corresponding
sequences of nested Markov partitions. Since $\|S_{m}-S\| \to 0$ as $m\to \infty$,
we have a constant $a=a(S)>0$ such that $S_{m}(w^{*}) \geq a$ for sufficient large $m$
and all $w^{*}\in \Sigma_{A}^{*}$. Let us assume this true for all $m$.
Since $\Sigma_{A}^{*}$ is a compact set, we have that there is another
constant $b=b(a)>0$ such that $S_{m}(w_{n}^{*}) \geq b$
(pre-scaling functions in Lemma~\ref{scalingfunction})
for all $m$ and all $n$. This says that the collection of the sequences
$\{\eta_{m,n}\}_{n=0}^{\infty}$ of nested Markov partitions
has uniformly bounded geometry. From a method in~\cite{Jiang4}, which gives a calculation of quasisymmetric dilatation
from bounded geometry, we have a constant $M>0$ such that the quasisymmetric dilatations of all $H_{m}$ are less than or equal to $M$. From~\cite{DouadyEarle}, we know the quasiconformal dilatation of the Douady-Earle extension of $H_{m}$ to ${\mathbb D}$ is controlled by the quasisymmetric dilatation of $H_{m}$. Thus we have a constant $K$ such that all $K_{m}$ is less than or equal to $K$.
This says that the sequence $\{\tau_{m}\}_{m=1}^{\infty}$ is contained in the closed ball
$$
B_{K} (S) =\{\eta\in T(R) \;|\; d_{T} (\eta, \tau) \leq \frac{1}{2} \log K\}
$$
which is a compact set. So we have a convergent subsequence. Let us assume that $\{\tau_{m}\}_{m=1}^{\infty}$
itself is convergent and converges to $\tilde{\tau}=[(Y, h_{Y})]$. Let
$$
h_{XY}= h_{X}\circ h_{Y}^{-1}: X={\mathbb D}/\Gamma_{X}\to Y ={\mathbb D}/\Gamma_{Y}
$$
be a $\tilde{K}=\exp(2d_{T} (\tau,\tilde{\tau}))$-quasiconformal homeomorphism. From our assumption,
we know that $\tilde{K}>1$. Let $H_{XY}$ be the corresponding boundary correspondence. Then $H_{m}$ converges to $H_{XY}$
on $S^{1}$ modulo M\"obius transformations as $m\to \infty$. Let us just assume that $H_{m}$
converges to $H_{XY}$ on $S^{1}$ as $m\to \infty$.

Let $f_{Y}$ be the corresponding Markov map and let $\{\eta_{n,Y}\}_{n=0}^{\infty}$ be the sequence
of nested Markov partitions. For any $w_{n}^{*}$, let $I_{w_{n}^{*}, X_{m}}\in \eta_{m,n}$
and $I_{w_{m}^{*}, Y}\in \eta_{n,Y}$. We have that $|I_{w_{n}^{*}, X_{m}}|\to |I_{w_{m}^{*}, Y}|$
as $m\to \infty$ for each fixed $n$ and $w_{n}^{*}$.

Since the sequences $\{\eta_{m,n}\}_{n=0}^{\infty}$ of nested Markov partitions have
uniformly bounded geometry, this again says that there are constants $C=C(S)>0$ and $0<\mu=\mu (S)<1$
such that $\nu_{n,m}\leq C\mu^{n}$ for all $n$ and $m$,
where
$$
\nu_{n,m} =\max_{I\in \eta_{n,m}} |I|.
$$
This implies that $S_{m}(w_{n}^{*}) \to S_{m}(w^{*})$ and $S_{Y}(w^{*}_{n}) \to S_{Y}(w^{*})$
as $n\to \infty$ uniformly on $m\geq 1$ and $w^{*}\in \Sigma_{A}^{*}$. Thus we can change double
limits for each $w^{*}\in \Sigma_{A}^{*}$,
$$
S (w^{*}) =\lim_{m\to \infty} S_{m}(w^{*}) = \lim_{m\to \infty} \lim_{n\to \infty} S_{m}(w_{n}^{*})
$$
$$
=\lim_{n\to \infty}\lim_{m\to \infty} S_{m}(w_{n}^{*}) =\lim_{n\to \infty} S_{Y}(w_{n}^{*}) = S_{Y}(w^{*}).
$$
From Theorem~\ref{samesf}, this implies that $\tilde{\tau}=\tau$, therefore, $d_{T}(\tilde{\tau},\tau)=0$.
This is a contradiction. The contradiction says that
$$
id_{MT}: (T(R), d_{max}(\cdot, \cdot) )\to (T(R),
d_{T}(\cdot, \cdot))
$$
is continuous at each point $S$. We have completed the proof.
\end{proof}

However, the map in the last theorem is general not uniformly continuous (actually all constants
in the proof depend on $S$). This can be examined by the union of graphs of $S\in T(R)$ which is an open section
in the open unit cube $\prod_{0}^{\infty}(0,1)$ and the maximum norm on this open unit cube is incomplete.
However, from Theorems~\ref{finermetric} and~\ref{semifinermetric}, we have that

\begin{corollary}
The topology on $T(R)$ induced from the maximum metric $d_{max}$ is the same as the topology on $T(R)$ induced from the usual Teichm\"uller metric $d_{T}$.
\end{corollary}

\section{\bf Added Remark: pressure metric and WP metric}

It is interesting to compare our function model and McMullen's thermodynamical
embedding. More interestingly, from McMullen's calculation in~\cite{McMullen}, we have that the pressure metric
for the tangent vector $d(\log S_{t})/dt|_{t=0}$ of a smooth path $\iota (\tau_{t})=S_{t}$
through $S_{0}$ is a constant times the Weil-Petersson metric of $d\tau_{t}/dt|_{t=0}$.

Consider the subshift of finite type $(\Sigma_{A}, \sigma_{A})$ in \S4 associated to all Marked Riemann surfaces $(X, h_{X})$ by $R$.
Let $C^{H}=C^{H}(\Sigma_{A})$ be the space of all H\"older continuous functions on $\Sigma_{A}$. Two functions
$\phi, \psi\in C^{H}(\Sigma_{A})$ are said to be cohomologously equivalent, denoted as $\phi\sim_{co} \psi$
if there is a continuous function $u$ on $\Sigma_{A}$ such that
$$
\phi -\psi = u\circ \sigma_A -u.
$$
It is an equivalence relation. We say $\phi$ is a co-boundary if $\phi\sim_{co} 0$. We use
$$
{\mathcal C}{\mathcal C}^{H}= C^{H}(\Sigma_{A})/ \sim_{co}
$$
to denote the space of all cohomologous equivalence classes. For each $\theta\in {\mathcal C}{\mathcal C}^{H}$,
there is an important thermodynamical quantity called the pressure $P(\theta)=P(\phi)$ for any $\phi\in\theta$ associated to it.
It is a smooth concave function on ${\mathcal C}{\mathcal C}^{H}$. Let
$$
{\mathcal C}{\mathcal C}^{H}_{0} =\{ \theta \in {\mathcal C}{\mathcal C}^{H}\;|\; P(\theta)=0\}
$$
be the subspace of all equivalence classes with zero pressure. In~\cite{McMullen}, McMullen embedded the Teichm\"uller space $T(R)$ into
${\mathcal C}{\mathcal C}^{H}_{0}$ through cohomologous equivalence classes $\theta=[\phi_{X}]$
where
$$
\phi_{X} =-\log f_{X}'\circ \pi_{X}
$$
and where $f_{X}$ are Markov maps associated to all marked Riemann surfaces $(X, h_{x})$ by $R$ in \S3.
Just like we did in the circle expanding mappings case in~\cite{Jiang5} (also see~\cite{Jiang6}), the scaling function $S_{\tau}$
on the dual symbolic space $\Sigma^{*}_{A}$ can be thought as a single function representation for the cohomologous equivalence
class $\theta=[\phi_{X}]$ for any marked Riemann surface $(X, h_{X})\in \tau$ but it is in the dual point of view.
Therefore, our scaling function model ${\mathcal F}$ can be thought as a dual version of McMullen's thermodynamical embedding.
However, our model gives a single function representation for each coholomogous equivalence class as well as
for each Teichm\"uller equivalence class.

For every $\theta \in {\mathcal C}{\mathcal C}^{H}_{0}$,
there is a unique Gibbs measure $m_{\theta}$ for the system $(\Sigma_{A}, \sigma_{A}, \phi)$ where $\phi$
is any function in $\theta$ (see, for example,~\cite{Jiang7} and other references in it). For every $[\psi]\in {\mathcal C}{\mathcal C}^{H}_{0}$ with zero mean,
that is, $\int_{\Sigma_{A}} \psi dm_{\phi}=0$, the variance is given by
$$
Var ([\psi], m_{\phi}) = \lim_{n\to \infty} \frac{1}{n} \int_{\Sigma_{A}} \Big|\sum_{k=0}^{n-1} \psi\circ \sigma_{A}^{k}(w)\Big|^{2} dm_{\phi}.
$$
Then by convexity, the second derivative
$$
D^{2} P([\psi]) = Var ([\psi], m_{\theta}).
$$
The pressure metric of $[\psi]$, given by
$$
||[\psi]||^{2}_{P} =\frac{Var ([\psi], m_{\theta})}{-\int_{\Sigma_{A}} \phi dm_{\theta}},
$$
is nondegenerate. Suppose $\tau_{t}$ is a smooth path in $T(R)$ through $\tau_{0}$. The tangent vector $\dot{\tau_{0}}= d\tau_{t}/dt|_{t=0}$
can be represented uniquely by a harmonic Beltrami differential $\mu= \rho^{-2} \overline{\phi}$ where $\rho$ is the
hyperbolic metric and $\phi$ is a holomorphic quadratic differential. The Weil-Petersson metric on the
tangent space $T_{\tau_{0}} T (R)$ is given by
$$
|| \dot{\tau_{0}}||^{2}_{WP} = ||\mu||^{2}_{WP} =\int \rho^{2} |\mu|^{2} =\int \rho^{-2} |\phi|^{2}.
$$

Suppose $\tau_{t}$ is a smooth path in $T(R)$ through $\tau_{0}$.
There is a unique family of homeomorphisms $H_{t}$ of $S^{1}$ such that it is the family of boundary correspondences
from $\Gamma_{t}$ to $\Gamma_{0}$ where $X_{t}={\mathbb D}/\Gamma_{t}\in \tau_{t}$. Let $\Phi_{t}$ be the family
of quasiconformal homeomorphisms from Bers' embedding in \S6. Let $\Lambda_{t}$ be the image of $S^{1}$
under $\Phi_{t}$. Then $\Lambda_{t}$ is a quasicircle and is the limit set of the quasi-Fuchsian group $\tilde{\Gamma}_{t}$
obtained by gluing the unit disk and the outer of unit disk by $H_{t}$. Let $a(t)=HD(\Lambda_{t})$ be the Hausdorff dimension
of $\Lambda_{t}$. Then $a(t)$ has the minimum value $1$ at $t=0$ since $\Lambda_{0}=S^{1}$. Let $m_{t}=H_{t*}Leb$ be the
pushforward measure of Lebesgue measure on $S^{1}$ by $H_{t}$. Then from Theorem~\ref{Tukia}, $m_{t}$ is totally singular
with respect to Lebesgue measure. Let $b(t)=HD(m_{t})$ be the Hausdorff dimension of the measure $m_{t}$, that is,
$$
b(t) = \inf \{ HD (E)\;|\; m_{t}(E)=1\}.
$$
Then $b(t)$ has the maximum value $1$ at $t=0$.
Let $\theta_{t}$ be the corresponding cohomologous equivalence classes to $\tau_{t}$ from McMullen's thermodynamical embedding.
Using a key equality in thermodynamical formalism,
$$
P([\phi]+t[\psi]) = P([\phi]) +\frac{t^{2}}{2} Var ([\psi], m_{[\phi]}) +O(t^{3}),
$$
where $m_{[\phi]}$ is the Gibbs measure for the system $(\Sigma_{A}, \sigma_{A}, \phi)$ and $[\psi]$ has zero mean and
$Var ([\psi], m_{[\phi]})$ is the variance, McMullen proved that
$$
\frac{1}{4} || \dot{\theta_{0}}||^{2}_{P} = \frac{d^{2} a(t)}{dt^{2}}|_{t=0} =- \frac{1}{4}\frac{d^{2} b(t)}{dt^{2}}|_{t=0} = \frac{1}{3} \frac{||\dot{\tau_{0}}||^{2}_{WP}}{area(\tau_{0})}.
$$

Now let us consider the dual symbolic dynamical system $(\Sigma_{A}^{*}, \sigma_{A}^{*})$ and the space of all functions $\log S_{\tau}$ for $\tau\in T(R)$ and $\iota (\tau) =S_{\tau}$. First we have that the pressure $P(\log S_{\tau})=0$ for every $S_{\tau}$. Let $S_{t}=\iota (\tau_{t})$ be the corresponding smooth path through $S_{0}$ in our function model. Let $m^{*}_{0}=m^{*}_{\log S_{0}}$ be the Gibbs measure for the system $(\Sigma_{A}^{*}, \sigma_{A}^{*}, \log S_{0})$. From the fact that $P(\log S_{t}) =0$, we have that
$$
\frac{dP(\log S_{t})}{dt}\Big|_{t=0}=\int_{\Sigma_{A}^{*}} \frac{d(\log S_{t})}{dt}\Big|_{t=0} dm^{*}_{0} =0.
$$
Thus the vector $d\log S_{t}/dt|_{t=0}$ has zero mean. The variance is then can be calculated as
$$
Var \Big( \frac{d(\log S_{t})}{dt}\Big|_{t=0}, m^{*}_{0}\Big) = \lim_{n\to \infty} \frac{1}{n} \int_{\Sigma_{A}^{*}} \Big|\sum_{k=0}^{n-1} \frac{d(\log S_{t})}{dt}\Big|_{t=0} \circ (\sigma_{A}^{*})^{k}(w^{*})\Big|^{2} dm^{*}_{0}.
$$
The pressure metric for $d(\log S_{t})/dt|_{t=0}$ can be then defined and is given by
$$
\Big\|\frac{d(\log S_{t})}{dt}\Big|_{t=0}\Big\|^{2}_{P} =\frac{Var \Big( \frac{d(\log S_{t})}{dt}\Big|_{t=0}, m^{*}_{0}\Big)}{-\int_{\Sigma^{*}_{A}} \log S_{0} dm^{*}_{0}}.
$$

From~\cite[page 76-77]{Jiang6}, for each periodic point $w^{*}= (j_{n-1}\cdots j_{0})^{\infty}$ of $\sigma_{A}^{*}$, we have a periodic point $w=(i_{0}\cdots i_{n-1})^{\infty}$ of $\sigma_{A}$. This correspondence of periodic points is bijective. Moreover, from~\cite[Proposition 3.3]{Jiang6},
$$
\sum_{k=0}^{n-1} \log S_{\tau} ((\sigma_{A}^{*})^{k} (w^{*})) =\sum_{k=0}^{n-1} \phi_{X} (\sigma_{A}^{k} (w)).
$$
for any $(X, h_{X})\in \tau\in T(R)$. Since the pressures $P([\phi_{X}])$ and $P(\log S_{\tau})$ only depend
on summations of values over periodic cycles, so they are equal, that is,
$$
P(\log S_{\tau})= P([\phi_{X}]).
$$
Just like we did in~\cite{Jiang5}, there is a one-to-one correspondence between Gibbs measures $m_{[\phi_{X}]}$ for systems
$(\Sigma_{A}, \sigma_{A}, \phi_{X})$ and Gibbs measure $m_{\log S_{\tau}}^{*}$ for systems
$(\Sigma_{A}, \sigma_{A}, \log S_{\tau})$. Thus for the smooth curve $\{\iota (\tau_{t}) = S_{t}\}$, the variance
$$
 Var\Big(\frac{d(\log S_{t})}{dt}\Big|_{t=0}, m^{*}_{\log S_{0}}\Big)=Var(\dot{\theta_{0}}, m_{\theta_{0}}).
$$
Moreover,
$$
\int_{\Sigma^{*}_{A}} \log S_{0} dm^{*}_{0}=\int_{\Sigma_{A}} \phi_{0} dm_{\theta_{0}}
$$
for any $\phi_{0}\in \theta_{0}$.
Thus we have that the pressure metric
$$
\Big\| \frac{d(\log S_{t})}{dt}\Big|_{t=0}\Big\|_{P}^{2} = \frac{4}{3} \frac{||\dot{\tau_{0}}||^{2}_{WP}}{area(\tau_{0})}.
$$

\vspace*{20pt}
\bibliographystyle{amsalpha}

\end{document}